\documentclass[11pt]{article}
\usepackage[a4paper,margin=1in]{geometry}
\usepackage{amsmath,amssymb,amsthm}
\usepackage{bm}
\usepackage{graphicx}
\usepackage{booktabs}
\usepackage{mathtools}
\usepackage{float}
\usepackage[hidelinks]{hyperref}

\usepackage{natbib}
\setcitestyle{authoryear,round,open={(},close={)}}
\usepackage{xcolor}
\usepackage{hyperref}
\usepackage[shortlabels]{enumitem}
\hypersetup{colorlinks=true,linkcolor=blue,citecolor=blue,urlcolor=blue}

\usepackage[ruled,vlined,linesnumbered]{algorithm2e}
\SetKwInput{Input}{Input}
\SetKwInput{Output}{Output}

\newtheorem{theorem}{Theorem}
\newtheorem{lemma}[theorem]{Lemma}
\newtheorem{proposition}[theorem]{Proposition}
\newtheorem{corollary}[theorem]{Corollary}
\theoremstyle{definition}
\newtheorem{definition}[theorem]{Definition}
\theoremstyle{remark}
\newtheorem{remark}[theorem]{Remark}

\usepackage{amsthm}

\long\def\amirali#1{{\color{black}#1}}

\long\def\aaa#1{{\color{black}#1}}

\long\def\ghmr#1{{\color{black}#1}}
\long\def\aaar#1{{\color{black}#1}}
\long\def\ghmra#1{{\color{black}#1}}

\title{On Approximate Computation of Critical Points}
\author{Amir Ali Ahmadi\thanks{Amir Ali Ahmadi is with the Department of Operations Research and Financial Engineering at Princeton University. Email: \href{mailto:aaa@princeton.edu}{\texttt{aaa@princeton.edu}}\\
Amir Ali Ahmadi was partially supported by the Sloan Fellowship, the Princeton AI Lab Seed Grant, the Princeton SEAS Innovation Grant, and a Research Gift in Mathematical Optimization.} \and Georgina Hall\thanks{Georgina Hall is with the Decision Sciences Area at INSEAD. Email: \href{mailto:georgina.hall@insead.edu}{\texttt{georgina.hall@insead.edu}}}}
\date{}

\begin{document}
\maketitle

\begin{abstract}
We show that computing even coarse approximations of critical points is intractable for simple classes of nonconvex functions. \ghmra{We first present a family of $n$-variate cubic polynomials with coefficients in $\{-1,0,1\}$ that have exact critical points but for which no polynomial-time algorithm can write down a $\sqrt{n}/4$-approximate critical point in its standard representation. We then show that computational hardness persists when an exact critical point has a short representation. More specifically, we prove that for every fixed $\eta>0$, unless P $=$ NP, no polynomial-time (or even pseudo-polynomial-time) algorithm can take an $n$-variate cubic polynomial with rational coefficients ${c_\alpha}$ and return a point with gradient norm at most $(\max_\alpha |c_\alpha|)n^{1/2-\eta}$ whenever the polynomial has a critical point in $\{-1,0,1\}^n$. The algorithm may return an arbitrary point on all other inputs. Both results are essentially tight in their dependence on the dimension $n$. We also prove hardness results for lower-bounded functions and for an alternative notion of approximation, in which approximation is measured by distance to an exact critical point.} Overall, our results stand in contrast to the commonly held belief that, in nonconvex optimization, approximate computation of critical points is a tractable task.
\end{abstract}


\section{Introduction} \label{sec:intro}
\aaa{A \emph{critical point} (or a \emph{stationary point}) of a differentiable function is a point at which its gradient vanishes. The problem of computing a critical point is one of the most fundamental algorithmic tasks in mathematical optimization. One key reason for this is that any global (or even local) minimizer of a differentiable function must necessarily be a critical point.
Therefore, the task of finding a critical point is often viewed as a prerequisite to that of finding minimizers. Unsurprisingly, structural properties of a function---such as convexity or the Polyak-Łojasiewicz property---that make this task sufficient for finding global minimizers often gain prominence in optimization theory.  }

\aaa{Over the years, numerous algorithms have been proposed in the optimization literature to search for critical points under various assumptions; } see, e.g., \citep{bertsekas1997nonlinear,boyd2004convex,nocedal2006numerical,cartis2022evaluation} and references therein. Perhaps the two best-known such algorithms are gradient descent, which forms the computational backbone of much of modern machine learning, and Newton's method, which is a key component of interior point methods. Naturally, critical points are fixed points for both of these algorithms. In the case of Newton’s method, the aim of each iteration is, in fact, also to find a critical \amirali{point, namely that of the degree-two Taylor approximation to the original function at the current iterate.}

Critical points---and more generally, local notions of optimality---have taken on even greater prominence in the optimization community in recent years, due to the paradigm shift from convex to nonconvex optimization, driven largely by machine learning. Indeed, in the nonconvex setting, it is well understood that global minimization is \amirali{often} an NP-hard problem. A common misconception, however, is that this hardness stems primarily from the existence of spurious local minima, i.e., local minima that are not global and that can trap optimization algorithms. This view overlooks the fact that \aaa{finding even a local minimum is an NP-hard task already for low-degree polynomials~\citep{ahmadi2022complexityqp}. The same complexity statement holds for recognizing whether a point is a local minimum~\citep{murty1985some}.}
Among researchers aware of the hardness results for local optimality, a commonly heard justification for the difficulties encountered is the presence of saddle points or other critical points that hinder algorithmic progress. This is well illustrated, e.g., by a quote from\amirali{~\citep{ge2015escaping}}: 
\begin{quote}
``The difficulty [of optimizing a nonconvex function] comes from two aspects. First, a nonconvex function may have many local minima and it might be hard to find the best one (global minimum) among them. Second, even finding a local minimum might be hard as there can be many saddle points which have $0$-gradient but are not local minima.'' 
\end{quote}

A widespread belief reflected in this quote is that finding a critical point should be relatively straightforward--perhaps because critical points are easy to recognize, as the definition requires checking equations at a single point rather than an inequality over a neighborhood (as is the case for local minima), or perhaps because the aforementioned algorithms are expected to converge in a reasonable number of iterations.

Running counter to this common belief, it was recently shown\amirali{,} via a \amirali{relatively simple polynomial-time reduction}
from the maximum-cut problem in combinatorial optimization\amirali{,} that deciding if a polynomial of degree 3 or higher has a critical point is NP-hard~\citep[Theorem 2.1]{ahmadi2022complexity}.\footnote{The same decision problem for degree-2 polynomials is \amirali{in P,} as it simply requires testing feasibility of a linear system of equations.} It follows that unless P=NP, an \emph{exact} critical point of a low-degree polynomial cannot be found by polynomial-time algorithms. \ghmra{To reconcile this result with the prior belief, it is natural to assume that the difficulty likely stems from the need for exact computation.} In practice, however, one is often content with \emph{approximate} computation of a critical point, which is thought to be a much simpler task. Our goal in this paper is to examine this assumption. As it turns out, our findings go against this intuition.

\aaa{\subsection{Outline, contributions, and related literature}}

\ghmra{Following the preliminaries in Section~\ref{sec:preliminaries}, we first study, in Section~\ref{sec:rep.comp}, a representational obstruction to the approximate computation of critical points. We consider the notion of \emph{$\epsilon$-approximate critical points} of a scalar-valued function $f$, that is, points where the 2-norm of the gradient of $f$ is no more than $\epsilon$. In Theorem~\ref{thm:representation.approx.cubic}, we construct a simple family of $n$-variate cubic polynomials with coefficients in $\{-1,0,1\}$ that have a critical point, but for which every rational $\epsilon$-approximate critical point with $\epsilon$ as large as a constant multiple of $\sqrt{n}$ requires exponentially many bits in its natural representation. 

In Section~\ref{sec:comp.comp}, we show that the problem of computing an approximate critical point remains intractable even when the representational obstruction identified above is absent. Throughout, we work in the standard Turing model of computation, where the function $f$ belongs to a specific parametric function class (e.g., low-degree polynomials), which we specify in the formal statements of the theorems. In Theorem~\ref{thm:eps.approx.cubic}, we show that, unless P=NP, computing an approximate critical point of an $n$-variate cubic polynomial $p$ with coefficients $\{c_{\alpha}\}$ is intractable even when one is content with gradient norm as large as $\max_{\alpha} |c_{\alpha}|n^{1/2-\eta}$ for any fixed $\eta>0$, and even when $p$ is guaranteed to have a critical point with a simple representation. Showing this hardness of approximation result is considerably more involved than showing hardness of exact computation of a critical point as done in~\citep[Theorem~2.1]{ahmadi2022complexity}. We next investigate whether this hardness persists when the function is known to be lower bounded. We first observe in Proposition~\ref{prop:crit.lb} that deciding whether a lower bounded function has a critical point is NP-hard. In Proposition~\ref{thm:critical.point.spurious}, we show that the existence of a polynomial-time algorithm for finding a critical point of a lower bounded function with no spurious critical points (and hence no spurious local minima) would refute the Planted Clique conjecture and imply PPAD=P. In Theorem~\ref{thm:lower-bounded-exponential-radius}, we show that, unless P=NP, computing an approximate critical point of a lower bounded $n$-variate sextic polynomial $p$ with coefficients $\{c_{\alpha}\}$ is intractable even when $p$ is guaranteed to have a critical point in $\{-1,0,1\}^n$, and one is content with finding a $\left( \max_{\alpha} |c_{\alpha}|n^{1/2-\eta}\right)$-approximate critical point, for any fixed $\eta>0,$ in a ball of radius $2^n.$ Our final result considers the alternative notion of \emph{$\epsilon$-near critical points} of a function $f$, which are points that are within distance $\epsilon$ from some critical point of $f$. In Theorem~\ref{thm:near.critical}, we show that hardness persists under this notion even for lower bounded functions and even when one allows the distance to a critical point to grow nearly as $\sqrt{n}$. As it turns out, the approximation thresholds in our hardness results are essentially tight in their dependence on the dimension $n$. }

We end our introduction by reviewing related literature on hardness of computation of approximate critical points. We are not aware of any results in the literature regarding hardness of computation of near critical points. In \citep{carmon2020lower}, the authors study an oracle-based complexity model, \ghmra{sometimes also called the black-box model}, in which access to a function is provided through evaluations of the function and its derivatives up to a certain order at queried points. They establish lower bounds on the number of queries required by any algorithm in this model to compute an $\epsilon$-approximate critical point. In \citep{carmon2021lowerii}, the authors derive analogous lower bounds under the additional assumption of convexity. In both of these papers, the hard instances constructed by the authors have a domain whose dimension grows with $\frac{1}{\epsilon}$. In \citep{hollender2025computational,chewi2023complexity}, the authors also provide lower bounds on the number of queries needed in the oracle model to recover an $\epsilon$-approximate critical point, but they do so in fixed and low dimension. \ghmra{These results are information theoretic lower bounds for algorithms that access the function only through oracle queries; in particular, they do not rule out algorithms that exploit an explicit description of the function, or account for the amount of computation an algorithm performs with the query responses.}

\ghmra{Computational hardness of finding approximate critical points has also been studied in the ``white-box'' model.} In this model, one is given access to Turing machines which can compute, given an input point $x$, the evaluation of a function $f$ or of its gradient at $x$ in time polynomial in the size of $x$. \ghmra{In \citep{daskalakis2021complexity}, the authors consider smooth and Lipschitz functions over the domain $[0,1]^n$ that are represented in this manner. They show that the search problem of finding an $\epsilon$-approximate critical point, even for a fixed constant value of $\epsilon$, is complete for the complexity class Functional NP (FNP).
In \citep{hollender2025computational}, the authors consider the same white-box representation \aaar{of functions} and show that, in any fixed dimension $n\geq 2$, finding an approximate critical point of a smooth lower bounded function over an unrestricted domain} is complete for \aaar{the complexity class} polynomial local search (PLS), which lies between \amirali{the} functional versions of P and NP.  \ghmra{Extending this line of work, \citet{kontogiannis2024computational} consider the same white-box model and show that finding an approximate second-order stationary point of a bounded two-dimensional function with Lipschitz gradient and Hessian over an unrestricted domain is PLS-complete.}

\ghmra{The complexity results in our paper are established in the Turing model of computation and in a natural setting that is more familiar to the optimization community. We focus on finitely parametrized classes of functions (most often, polynomials of fixed degree), so the input to a candidate algorithm consists simply of the rational numbers specifying these parameters (e.g., the coefficients of a polynomial). Of course, this setting can be viewed as a special case of the white-box model, since the white box can, for example, represent a circuit that, given the coefficients of a polynomial, evaluates the polynomial and its gradient at any specified point. Our hardness results show that even for this substantially more structured subclass of the white-box model—corresponding to simple function families that commonly arise in optimization—a polynomial-time algorithm for approximately computing critical points would imply P=NP.}
\ghmra{Our results} rule out not only methods based on function or derivative evaluations, such as gradient descent, but also any computational approach that exploits the explicit polynomial representation, including, for example, methods based on convex relaxations followed by rounding procedures. We note that hardness results in fixed dimension, as obtained \aaar{e.g.} by \citep{hollender2025computational}, are not possible in our setting, \amirali{where the function class of interest is polynomials.}
Indeed, if the dimension is fixed, \amirali{even the more general problem of solving} polynomial systems of equations and inequalities can be \amirali{accomplished} \amirali{using} quantifier elimination theory and related algorithms (see, e.g.,~\citep{grigor1988solving}) in time polynomial in their degrees and the size of their coefficients.
Finally, we remark that there are hardness results on computation of Karush-Kuhn-Tucker points, which are analogs of critical points for constrained optimization problems \citep{vavasis1993black,fearnley2022complexity,fearnley2025complexity}. \ghmra{In the presence of constraints, such hardness results can already arise with quadratic objective functions \citep{fearnley2025complexity}.}  \\

\ghmra{\section{Preliminaries} \label{sec:preliminaries}}

Throughout the paper, we denote the gradient vector of a scalar-valued function $f$ by $\nabla f $.\\

\begin{definition}[Critical point]
Let $f:\mathbb{R}^n\rightarrow \mathbb{R}$ be a differentiable function. A \aaa{\emph{critical point}} of $f$ is a point $x \in \mathbb{R}^n$ such that $\nabla f(x)=0.$ 
\end{definition}

The paper considers two different notions of approximate computation of a critical point: we mainly focus on \emph{$\epsilon$-approximate} critical points, but also consider an alternative notion in Section \ref{sec:comp.comp}, \emph{$\epsilon$-near} critical points. We define both notions next. Throughout the paper, the notation $||\cdot||$ signifies the  2-norm.

\begin{definition}[\aaa{$\epsilon$-approximate critical point}]
\aaa{Let $f:\mathbb{R}^n\rightarrow \mathbb{R}$ be a differentiable function and \amirali{$\epsilon>~0$}}. An \emph{$\epsilon$-approximate critical point} of $f$ is a point $x \in \mathbb{R}^n$ such that $||\nabla f(x)|| \leq \epsilon.$
\end{definition}

\begin{definition}[$\epsilon$-near critical point]
Let $f:\mathbb{R}^n\rightarrow \mathbb{R}$ be a differentiable function and $\epsilon>0$. An \emph{$\epsilon$-near critical point} of $f$ is a point $x \in \mathbb{R}^n$ such that $||x-x^*||\leq \epsilon$ for some critical point $x^*$ of $f$.
\end{definition}

The notion of a near critical point is a very natural alternative to that of an approximate critical point.
Indeed, a critical point $x^*$ is an optimal solution to the optimization problem $\min_x ||\nabla f(x)||.$ An approximate critical point thus is a point that approximates the \emph{optimal value} of this problem, whereas a near critical point is an approximation of an \emph{optimal solution} of this problem. \ghmra{These two notions are similar, e.g., to two different ways of defining approximate Nash equilibria; see \citet{etessami2010complexity} and the discussion around weak and strong fixed points.}

As mentioned in the introduction, all of the computational complexity results in this paper are formulated in the standard Turing model of computation. In this model, every instance of the problem has an input representable with a finite number of bits (see, e.g.,~\citep{sipser1996introduction}). Therefore, we consider differentiable functions that are finitely parametrized. Most of our main theorems involve polynomials of fixed degree (Theorems~\ref{thm:eps.approx.cubic}, \ref{thm:lower-bounded-exponential-radius}). Some other results (Theorem~\ref{thm:near.critical} and Propositions~\ref{prop:crit.lb}, \ref{thm:critical.point.spurious}) involve a mixture of polynomials of fixed degree and exponentials of polynomials of fixed degree. \ghmra{Each coefficient of the monomials that appear in these polynomials is represented as $a/b$ where $a$ is a signed integer and $b$ is a positive integer. The \emph{size} of a rational number $a/b$ is the number of bits needed to encode it in binary including its sign, i.e., $1+\lceil \log_2 |a|+1 \rceil + \lceil \log_2 |b|+1 \rceil$. The  \emph{input size} $\mathcal{L}$ is the sum of the sizes of all monomial coefficients. Let $\mathcal{M}$ be the maximum of $|a|$ and $b$ over all monomial coefficients. An algorithm is \emph{polynomial-time} if its running time is
bounded by a polynomial in $\mathcal{L}$, and \emph{pseudo-polynomial-time} if its running time is bounded by a polynomial in $\mathcal{M}$ and the number of coefficients. In particular, every polynomial-time algorithm is also pseudo-polynomial-time, though the converse is not true as pseudo-polynomial time algorithms can take exponential time in $\mathcal{L}$. }

\ghmra{\section{Representation Obstruction to Polynomial-Time Algorithms} \label{sec:rep.comp}

In this section, we ask whether, for a desired approximation accuracy, polynomials with rational coefficients always admit approximate critical points of polynomial size, so that a polynomial-time algorithm could write them down. Note that already for univariate cubics, rational coefficients do not guarantee rational critical points: for example, $x^3-x$ has its critical points at $\pm \frac{\sqrt{3}}{3}$. Nevertheless, by continuity of $\nabla p$ and density of $\mathbb{Q}^n$ in $\mathbb{R}^n$, any polynomial $p$ that has a critical point also has a rational $\epsilon$-approximate critical point for every $\epsilon>0$. The question is whether such a point can always be chosen to have polynomial size. Our next theorem answers this question negatively, even when $\epsilon$ is quite large.

\begin{theorem} \label{thm:representation.approx.cubic}
    There exists a family of cubic polynomials $\{p_n\}_{n\geq 4}$ in $n$ variables such that
    \begin{enumerate}[(i)]
        \item all coefficients of $p_n$ belong to $\{-1,0,1\}$,
        \item $p_n$ has at least one critical point,
        \item every (rational) $\frac14 \sqrt{n}$-approximate critical point of $p_n$ has size $2^{\Omega(n)}.$
        \end{enumerate}
\end{theorem}

It is easy to see that for any $n$-variate cubic polynomial $p$ with coefficients in $\{-1,0,1\}$ the origin is an $\sqrt{n}$-approximate critical point. Indeed, letting $c_i$ be the coefficient of the monomial $x_i$, and $c\mathrel{\mathop{:}}=(c_1,\ldots,c_n)^T$, we have $||\nabla p(0)||\leq ||c||\leq \sqrt{n}.$ Thus, up to a constant factor, Theorem \ref{thm:representation.approx.cubic} provides a \emph{tight} threshold for approximation below which simply writing down an approximate critical point can take exponential time. 

\begin{proof}
Fix $n\geq4$ and let $m=\left\lfloor\frac n2\right\rfloor$.
For $x=(x_1,\ldots,x_n)$, define
\[
p_n(x)
=
x_{m+1}(x_1-1)
+
\sum_{i=1}^{m-1}
x_{m+i+1}\bigl(x_{i+1}-x_i^2-1\bigr).
\]
Note that $p_n$ is $n$-variate and has degree 3. Furthermore,
it evidently satisfies condition $(i)$. To see that condition
$(ii)$ holds, we first compute the derivatives of $p_n$:
\[
\frac{\partial p_n(x)}{\partial x_{m+1}}=x_1-1,
\qquad
\frac{\partial p_n(x)}{\partial x_{m+i+1}}
=x_{i+1}-x_i^2-1
\quad\text{for }i=1,\ldots,m-1,
\]
and
\[
\frac{\partial p_n(x)}{\partial x_i}
=x_{m+i}-2x_{m+i+1}x_i
\quad\text{for }i=1,\ldots,m-1,
\qquad
\frac{\partial p_n(x)}{\partial x_m}=x_{2m}.
\]
All remaining partial derivatives are zero.
Now, set
\[
x_1^*=1,\quad
x_{i+1}^*=\left(x_i^*\right)^2+1
\text{ for }i=1,\ldots,m-1,\quad
\text{and}\quad
x_j^*=0
\text{ for }j=m+1,\ldots,n.
\]
It is easy to see that $\nabla p_n(x^*)=0$ and therefore
$x^*$ is a critical point of $p_n$.

Finally, we show that $(iii)$ holds. Let $x\in\mathbb{Q}^n$
be a $\frac{\sqrt n}{4}$-approximate critical point of $p_n$,
i.e., $\|\nabla p_n(x)\|\leq\frac{\sqrt n}{4}$.
Define $r_1=x_1-1$ and
$r_{i+1}=x_{i+1}-x_i^2-1$, for $i=1,\ldots,m-1$.
From above, we have
\begin{equation}\label{eq:ub.ri}
\sum_{i=1}^m r_i^2
=
\sum_{i=1}^m
\left(\frac{\partial p_n(x)}{\partial x_{m+i}}\right)^2
\leq
\|\nabla p_n(x)\|^2
\leq
\frac{n}{16}.
\end{equation}

Let $k=\left\lfloor\frac{3m}{4}\right\rfloor$.
We now show that: (i') $x_{k+1}\geq\frac{m}{8}$ and
(ii') for large enough $m$, $x_i\geq\frac{m}{8}$
for $i=k+1,\ldots,m$, and
$x_{i+1}\geq\frac12 x_i^2$ for $i=k+1,\ldots,m-1$.
We first show (i'). Since $x_1=1+r_1$ and
\[
x_{i+1}-x_i
=
x_i^2-x_i+1+r_{i+1}
=
\left(x_i-\frac12\right)^2+\frac34+r_{i+1}
\geq
\frac34+r_{i+1},
\]
we obtain
$x_{k+1}=x_1+\sum_{i=1}^{k}(x_{i+1}-x_i)
\geq1+\frac{3k}{4}+\sum_{i=1}^{k+1}r_i$.
By the Cauchy--Schwarz inequality and \eqref{eq:ub.ri},
\[
\sum_{i=1}^{k+1}r_i
\geq
-\sqrt{k+1}
\left(\sum_{i=1}^m r_i^2\right)^{1/2}
\geq
-\frac14\sqrt{n(k+1)}.
\]
Since $k\geq\frac{3m}{4}-1$, $k+1\leq m$,
$n\leq2m+1\leq3m$, it follows that
\[
x_{k+1}
\geq
1+\frac34\left(\frac{3m}{4}-1\right)
-\frac14\sqrt{3m^2}
=
\frac14+
\left(\frac{9}{16}-\frac{\sqrt3}{4}\right)m
\geq
\frac{m}{8}.
\]
We now show (ii') for $m\geq32$ by induction.
Obviously, $x_{k+1}\geq\frac{m}{8}$.
Suppose that $x_i\geq\frac{m}{8}$ for $i=k+1,\ldots,m-1.$
By \eqref{eq:ub.ri},
$|r_{i+1}|\leq\frac14\sqrt n\leq\frac14\sqrt{3m}$.
For $m\geq32$,
$\frac14\sqrt{3m}\leq\frac{m^2}{128}\leq\frac12x_i^2$.
Therefore,
\[
x_{i+1}
=
x_i^2+1+r_{i+1}
\geq
x_i^2-\frac14\sqrt{3m}
\geq
\frac12x_i^2.
\]
Moreover,
$\frac12x_i^2\geq\frac{m^2}{128}\geq\frac{m}{8}$.
This proves (ii').

Now, for $m$ large enough, define $u_i=\frac{x_i}{2}$
for $i=1,\ldots,m$.
From above, $u_{i+1}\geq u_i^2$ for $i=k+1,\ldots,m-1$,
and, since $m\geq32$, $u_{k+1}\geq\frac{m}{16}\geq2$.
Repeatedly applying the preceding inequality gives
\[
u_m
\geq
u_{k+1}^{\,2^{m-k-1}}
\geq
2^{2^{m-k-1}}.
\]
Since $k\leq3m/4$, $m-k-1\geq\frac{m}{4}-1$,
and hence $x_m=2u_m\geq2^{2^{m/4-1}}$.
As $x_m$ is rational, we write $x_m=a/b$ in irreducible
form, where $b\geq1$.
Since $x_m>0$, $a=bx_m\geq x_m$.
Thus the numerator of $x_m$ requires at least
$\log_2 x_m\geq2^{m/4-1}$ bits.
Since $m=\lfloor n/2\rfloor$, the numerator of $x_m$
requires at least $2^{\Omega(n)}$ bits,
which shows $(iii)$.
\end{proof}

}

\section{Computational Complexity of Computing Approximate Critical Points} \label{sec:comp.comp}

\ghmra{In the previous section, we identified a representational obstruction to polynomial-time algorithms for finding $\epsilon$-approximate critical points of low-degree polynomials: it can be the case that none exist with polynomial size. This can happen even when $\epsilon$ is as large as a constant multiple of the square root of the dimension. It is natural to wonder whether, in the absence of this representation issue (i.e., when restricted to problem instances where an $\epsilon$-approximate critical point of polynomial size exists), the \emph{search} for such a point can be carried out in polynomial time. In this section, we prove that unless P=NP, this cannot happen. In fact, our results show that the search problem remains hard even when there is an (exact) critical point in $\{-1,0,1\}^n$ and one is only hoping for an $\epsilon$-approximate critical point with $\epsilon$ rather large.}

\subsection{Complexity of computing an $\epsilon$-approximate critical point of a general function} \label{subsec:eps.approx.no.lb}

The following theorem is one of the main results of this paper.

\begin{theorem} \label{thm:eps.approx.cubic} \ghmra{Fix a scalar $\eta >0$.}
If there is an algorithm $\mathcal{A}$ that takes as input an $n$-variate cubic polynomial $p$ with coefficients $\{c_{\alpha}\}_{\alpha}$, runs in pseudo-polynomial time, and is such that
\begin{enumerate}[(i)]
    \item  when $p$ has no critical points \ghmra{in $\{-1,0,1\}^n$}, $\mathcal{A}$ is free to return an arbitrary point,
    \item when $p$ has a critical point \ghmra{in $\{-1,0,1\}^n$}, $\mathcal{A}$ returns a \ghmra{$\left(\max_{\alpha} |c_{\alpha}| \cdot n^{1/2-\eta}\right)$-approximate critical point,} 
\end{enumerate}
then $P=NP.$
\end{theorem}

\ghmra{

\begin{remark}
If $p$ has a critical point in $\{-1,0,1\}^n$, then it has, for every $\epsilon>0$, an $\epsilon$-approximate critical point of polynomial size: indeed, the critical point itself has this property. Thus, the computational hardness established in Theorem \ref{thm:eps.approx.cubic} does not arise from the difficulty of representing an $\epsilon$-approximate critical point, but from that of finding one.
\end{remark}
}
\ghmra{
\begin{remark} \label{rem:tight.with.n}
Two features of the approximation factor that appears in Theorem~\ref{thm:eps.approx.cubic} are worth noting. First, its dependence on the dimension $n$ is essentially tight. Indeed, as discussed after Theorem~\ref{thm:representation.approx.cubic}, we have $\|\nabla p(0)\| \leq \max_{\alpha}|c_\alpha|\sqrt{n}$, i.e., the origin is always a \(\max_{\alpha}|c_\alpha|\sqrt{n}\)-approximate critical point of \(p\). Theorem~\ref{thm:eps.approx.cubic} shows that, unless \(\mathrm{P}=\mathrm{NP}\), improving this trivial guarantee to $\max_{\alpha}|c_\alpha|\, n^{1/2-\eta},$ for any fixed \(\eta>0\), is computationally intractable. Second, the approximation factor is expressed relative to \(\max_{\alpha}|c_\alpha|\), rather than in absolute terms. This is because, for any differentiable function \(f\) and any \(\kappa>0\), $\|\nabla(\kappa f)(x)\|=\kappa\|\nabla f(x)\|,$ so the \(\epsilon\)-approximate critical points of \(f\) are exactly the \(\kappa\epsilon\)-approximate critical points of \(\kappa f\). Thus, if we wish to show hardness of computing approximate critical points over the class of all cubic polynomials, then, once a polynomial-time reduction yields hardness of computing an $\epsilon$-approximate critical point for $\epsilon$ constant, this threshold can be amplified to \(2^{poly(n)}\) by scaling the cubic polynomial by a factor of $2^{poly(n)}$, which preserves the polynomial-time property of the reduction (see Corollaries~\ref{cor:old.result} and \ref{cor:pseudo.poly}). However, in applications one may work with polynomial families whose coefficients are known a priori to be bounded. For this reason, we state our result in normalized form, relative to $\max_{\alpha}|c_\alpha|$.
\end{remark}

\begin{remark} \label{rem:eta}
The approximation factor in Theorem~\ref{thm:critical.point.spurious} is nonincreasing in $\eta.$ Consequently, an algorithm satisfying the stated guarantees for any fixed $\eta\geq \frac12$ would also satisfy them for $\eta=\frac14$, for example. It therefore suffices to prove the statement for $\eta \in (0,\frac12)$, which we assume throughout the proof, to avoid unnecessary technicalities. A similar remark holds for Theorems \ref{thm:lower-bounded-exponential-radius} and \ref{thm:near.critical}.
\end{remark}
}

To prove Theorem \ref{thm:eps.approx.cubic}, we give a polynomial-time reduction from the following decision problem known as \textsc{Clique}: Given a graph $G$ on $m$ vertices and a positive integer $k \leq m$, decide whether $G$ has a clique of size $k$, i.e., a set of $k$ pairwise adjacent vertices. It is \aaa{well known} that \textsc{clique} is NP-complete \citep{karp1975computational}, which implies in particular that, unless P=NP, this problem does not admit a polynomial-time algorithm.

Before presenting this reduction, we \aaa{prove} two \aaa{technical lemmas}. These involve showing properties of a \aaa{quartic polynomial $g \mathrel:{\mathbb{R}^\amirali{m}}\rightarrow\mathbb{R}$}, whose roots are in a one-to-one mapping with cliques in a graph (as formalized in Lemma~\ref{prop:clique.zeros.g}). This polynomial plays a crucial role in the proof of Theorem~\ref{thm:eps.approx.cubic} as well as in the results of Section~\ref{subsec:eps.approx.quartic}. We give its definition first. Let $G$ be a graph with $m$ vertices and adjacency matrix $A$ and let $x=(x_1,\ldots,x_m)$. \ghmra{For a parameter $k\in\{1,\ldots,m\}$, define} 
\begin{equation}\label{eq:def.g}
\begin{aligned}
g_{G,k}(x)=\left(\sum_{i=1}^m x_i-2k+m \right)^2 +m^2 \sum_{i=1}^m (x_i^2-1)^2 +\sum_i \sum_{j > i} (1-A_{ij})^2 (x_i+1)^2(x_j+1)^2.
\end{aligned}
\end{equation}
To simplify notation, we omit the dependence on $G,k$ and write simply $g$ \aaa{in \amirali{our} proofs.}

\begin{lemma}\label{prop:clique.zeros.g}
\aaa{A graph $G$ on $m$ vertices and with edge set $E$ has a clique of size $k$ if and only if the system of equations} 
\begin{equation} \label{eq:quad.eq}
\begin{cases}
&\sum_{i=1}^m x_i=2k-m\\
&x_i^2-1=0,~ i=1,\ldots,m\\
&(x_i+1)(x_j+1)=0,~1\leq i <j \leq m, ~\{i,j\}\notin \amirali{E}
\end{cases}
\end{equation} is feasible, or equivalently, \aaa{if and only if the polynomial $g=g_{G,k}$ in (\ref{eq:def.g}) vanishes at some point.}
\end{lemma}

\begin{proof}
It is easy to check that the only possible solutions to~(\ref{eq:quad.eq}) are indicator vectors of cliques of size $k$ in $G$ (that is, vectors in $\{\pm 1\}^m$ with exactly $k$ entries equal to $+1,$ whose indices correspond to vertices of $G$ forming a clique of size $k$).
\end{proof}

In the remainder of the paper, we let $sgn(x)$, for $x \in \mathbb{R}^m$, denote the vector in $\mathbb{R}^m$ whose $i$-th entry equals~$1$ if $x_i \ge 0$ and $-1$ if $x_i < 0$. Our next lemma provides a lower bound on the polynomial $g$ when $G$ does not have a clique of size $k$.

\ghmra{
\begin{lemma} \label{prop:lb.g.s}
Let $G$ be a graph on $m$ vertices and $k\in\{1,\ldots,m\}$. Recall the definition of $g=g_{G,k}$ in \eqref{eq:def.g}. Suppose $G$ has no clique of size $k$. Then,
$g(x) \geq 1, ~\forall x\in \mathbb{R}^m.$
\end{lemma}
}
\ghmra{
\begin{proof}
We prove the contrapositive. Let $x\in\mathbb{R}^m$ such that $g(x)<1.$
We show that $s=sgn(x) \in \{-1,1\}^m$ encodes a clique of size $k$ in $G$.

We first establish a bound on $||x-s||.$ As all terms in the definition of $g$ are nonnegative, $g(x)<1$ implies that $m^2\sum_{i=1}^m (x_i^2-1)^2<1$, or equivalently, $\sum_{i=1}^m (x_i^2-1)^2<\frac{1}{m^2}.$ 
For every $i$, we have
\[
|x_i-s_i|
=
\big||x_i|-1\big|
=
\frac{|x_i^2-1|}{|x_i|+1}
\le |x_i^2-1|,
\]
where the first equality follows from $s=sgn(x),$ the second follows from $$|x_i^2-1|=|~|x_i|-1| (|x_i|+1),$$ and the inequality follows from $1+|x_i|\geq 1.$ This, combined with the fact that $\sum_{i=1}^m (x_i^2-1)^2<\frac{1}{m^2}$, implies that
\[
\|x-s\|^2
=
\sum_{i=1}^m |x_i-s_i|^2
\le
\sum_{i=1}^m (x_i^2-1)^2
<
\frac{1}{m^2},
\]
which in turn implies that $\|x-s\|<\frac{1}{m}$.

We now show that $\sum_{i=1}^m s_i-2k+m=0$, i.e., exactly $k$ coordinates of $s$ are equal to one. Since $g(x)<1$, we have that $\left( \sum_{i=1}^m x_i-2k+m\right)^2 <1$, or equivalently, $\left|\sum_{i=1}^m x_i-2k+m\right|<1.$
By the triangle inequality and the Cauchy--Schwarz inequality,
\[
\begin{aligned}
\left|\sum_{i=1}^m s_i-2k+m\right|
\le
\left|\sum_{i=1}^m x_i-2k+m\right|
+
\left|\sum_{i=1}^m (s_i-x_i)\right| <
1+\sqrt{m}\,\|s-x\| <
1+\frac{1}{\sqrt{m}}\le 2.
\end{aligned}
\]
Now, let $r=|\{i~|~s_i=1\}|.$ Since $s_i\in\{-1,1\}$ for every $i$,
$\sum_{i=1}^m s_i=2r-m.$ Therefore $\sum_{i=1}^m s_i-2k+m=2(r-k),$
which is an even integer. Since its absolute value is strictly less than $2$, it must be zero, which implies $r=k$ and the claim.

It only remains to show that the vertices $\{i~|~s_i=1\}$ do indeed form a clique. We proceed by contradiction. Suppose that there exist distinct $i,j$ such that $s_i=s_j=1$ and $\{i,j\} \notin E.$ Since $s_i=s_j=1$, it must be the case that $x_i\geq 0$ and $x_j \geq 0$. Hence $(x_i+1)^2(x_j+1)^2\geq 1.$ As $\{i,j\}\notin E$, $A_{ij}=0$ and thus $g(x) \geq (1-A_{ij})^2(x_i+1)^2(x_j+1)^2 \geq 1$, which contradicts the fact that $g(x)<1.$ We deduce that the vertices $\{i~|~s_i=1\}$ form a clique. As $\{i~|~s_i=1\}$ is of cardinality $k$, it follows that $s$ encodes a clique of size $k$ in $G$.
\end{proof}
}

\ghmra{

\begin{proof}[Proof of Theorem~\ref{thm:eps.approx.cubic}.]
We prove that such an algorithm would solve \textsc{Clique} in polynomial time. Let $G$ be a graph on $m$ vertices\footnote{We require here that $m \geq 4.$ This is not a problem as \textsc{Clique} remains hard on instances of this type.} with adjacency matrix $A$ and let $k \in \{1,\ldots,m\}.$ We define $q(x)$ to be the vector of size $L =(m+1)+\frac{m(m-1)}{2}$ with entries $q_i(x)$ given by
\begin{equation} \label{eq:def.q}
\begin{aligned}
&q_1(x)=\sum_{i=1}^m x_i -2k+m, \quad q_2(x)=m(x_1^2-1),\quad \ldots, \quad q_{m+1}(x)=m(x_m^2-1),\\
&q_{m+2}(x)=(1-A_{12})(x_1+1)(x_2+1), \quad \ldots,\quad q_{L}(x)=(1-A_{(m-1)m})(x_{m-1}+1)(x_m+1).
\end{aligned}
\end{equation}
Now, let $K=m^{\lceil 3/\eta\rceil-2}$ (recall that we restrict ourselves in the proof to $\eta \in (0,1/2)$ so this quantity is well-defined; see Remark \ref{rem:eta}). For $\ell=1,\ldots,K$, we introduce the disjoint blocks of variables $w^{(\ell)} \in \mathbb{R}^L$. We then define the cubic polynomial
\begin{align}\label{eq:def.p}
p(x,w^{(1)},\ldots,w^{(K)})
=\sum_{\ell=1}^K (w^{(\ell)})^Tq(x)
\end{align}
in $n=m+KL$ variables. Note that the reduction from $(G,k)$ to $p$ can be carried out in polynomial time. As $\eta$ is fixed, $K=m^{c}$ for some constant $c$, and so $p$ has $n=m^{c'}$ variables for some constant $c'$. It thus has $m^{c''}$ coefficients for some constant $c''$, each of which is an integer. Moreover, the largest absolute coefficient of $p$ is $\gamma=m$, since the maximum coefficient of $q(x)$ is $m$ and distinct blocks share no monomials. Thus, each coefficient requires $O(\log(m))$ bits to encode, a property which is important for ruling out pseudo-polynomial time algorithms.

We prove that $(i')$ if there is a clique of size $k$ in $G$, then $p$ has a critical point in $\{-1,0,1\}^n$, $(ii')$ if there is no clique of size $k$ in $G$, then $\|\nabla p(v)\|^2\geq K$ for every $v\in\mathbb{R}^n$, and $(iii')$ $\gamma n^{1/2-\eta}<\sqrt K$. Once we have shown these claims, the theorem follows. To see this, let $\bar v \in \mathbb{R}^n$ be the point returned by algorithm $\mathcal{A}$.\footnote{Note that $\mathcal{A}$ would write this point down in time polynomial in $m$, since by assumption, it is a pseudo-polynomial time algorithm and the coefficients of $p$ are integers bounded in magnitude by $m$.} If $\|\nabla p(\bar v)\|^2<K$, then there is a clique of size $k$ in $G$ per claim $(ii')$. If $\|\nabla p(\bar v)\|^2\geq K$, then $p$ has no critical point in $\{-1,0,1\}^n$ (otherwise the algorithm would have returned a point with gradient norm at most $\gamma n^{1/2-\eta}<\sqrt K$, by claim $(iii')$). Per claim $(i')$, there is no clique of size $k$ in $G$. Since $\|\nabla p(\bar v)\|^2$ is rational and can be compared with $K$ exactly in polynomial time, if an algorithm like $\mathcal{A}$ existed, we would be able to solve \textsc{Clique} in polynomial time, which would imply $\mathrm{P}=\mathrm{NP}$.

To see claim $(i')$, note that the gradient of $p$, with respect to each block of variables, is
\[
\nabla_{w^{(\ell)}}p=q(x), \quad \ell=1,\ldots,K,
\qquad \text{ and } \qquad
\nabla_xp=\sum_{\ell=1}^K\sum_{j=1}^Lw_j^{(\ell)}\nabla q_j(x).
\]
From Lemma~\ref{prop:clique.zeros.g}, if there is a clique of size $k$ in $G$, \eqref{eq:quad.eq} is feasible. We let $\bar{x}$ be a solution to \eqref{eq:quad.eq}, that is, $q\left(\bar{x}\right)=0.$ Setting $\bar w^{(\ell)}=0$ for $\ell=1,\ldots,K$, we obtain
\[
\nabla p(\bar x,\bar w^{(1)},\ldots,\bar w^{(K)})=0.
\]
Furthermore, every coordinate of $(\bar x,\bar w^{(1)},\ldots,\bar w^{(K)})$ belongs to $\{-1,0,1\}$.

To see claim $(ii')$, suppose that there is no clique of size $k$ in $G$. Let $$v=(x,w^{(1)},\ldots,w^{(K)})\in\mathbb{R}^n.$$ Recalling the definition of $g=g_{G,k}$ in \eqref{eq:def.g}, we have
\begin{align} \label{eq:square}
\|\nabla p(v)\|^2
=\|\nabla_xp(v)\|^2+\sum_{\ell=1}^K\|\nabla_{w^{(\ell)}}p(v)\|^2\geq \sum_{\ell=1}^K ||\nabla_{w^{(\ell)}} p(v)||^2 \geq K\|q(x)\|^2=Kg(x)\geq K,
\end{align}
where the final inequality follows from Lemma \ref{prop:lb.g.s} as $G$ has no clique of size $k.$

Finally, to see claim $(iii')$, note that
\begin{align*}
\gamma n^{1/2-\eta} \leq m^{1+\lceil 3/\eta\rceil(1/2-\eta)} \leq m^{\lceil 3/\eta\rceil/2-2} <m^{\lceil 3/\eta\rceil/2-1}=\sqrt K,
\end{align*}
where we used $n=m+KL\leq K(m+L) \leq Km^2
\leq
m^{\lceil 3/\eta\rceil-2}\cdot m^2
=
m^{\lceil 3/\eta\rceil}$ for $m\geq 4$ and $\gamma=m$ for the first inequality, $\eta\lceil 3/\eta\rceil\geq 3$ for the second, and the definition of $K$ for the equality.
\end{proof}
}

\ghmra{In the proof of Theorem \ref{thm:eps.approx.cubic}, if we scale the cubic polynomial $p$ in \eqref{eq:def.p} by $2^n$ or $n^d$ respectively, we can obtain the following two corollaries.

\begin{corollary} \label{cor:old.result}
    If there is an algorithm $\mathcal{A}$ that takes as input an $n$-variate cubic polynomial $p$, runs in polynomial time, and is such that
\begin{enumerate}[(i)]

    \item  when $p$ has no critical points in $\{-1,0,1\}^n$, $\mathcal{A}$ is free to return an arbitrary point,

    \item when $p$ has a critical point in $\{-1,0,1\}^n$, $\mathcal{A}$ returns a $2^{n}$-approximate critical point,

\end{enumerate}

then $P=NP.$
\end{corollary}

\begin{corollary} \label{cor:pseudo.poly}
Fix a positive integer $d$.
    If there is an algorithm $\mathcal{A}$ that takes as input an $n$-variate cubic polynomial $p$, runs in pseudo-polynomial time, and is such that
\begin{enumerate}[(i)]

    \item  when $p$ has no critical points in $\{-1,0,1\}^n$, $\mathcal{A}$ is free to return an arbitrary point,

    \item when $p$ has a critical point in $\{-1,0,1\}^n$, $\mathcal{A}$ returns a $n^d$-approximate critical point,

\end{enumerate}

then $P=NP.$
\end{corollary}

}

\subsubsection*{Discussion around Theorem~\ref{thm:eps.approx.cubic}.} As a concrete example of the implications of Theorem~\ref{thm:eps.approx.cubic}, consider the gradient descent algorithm. Theorem~\ref{thm:eps.approx.cubic} establishes that, \ghmra{for any fixed $\eta >0$,} unless $\mathrm{P}=\mathrm{NP}$, there are $n$-variate cubic polynomials \ghmra{with coefficients $\{c_{\alpha}\}_{\alpha}$ that have critical points of small size, but} for which gradient descent, using any efficiently computable stepsize, does not return a \ghmra{$\left(\max_{\alpha} |c_{\alpha}| \cdot n^{1/2-\eta}\right)$}-approximate critical point in a polynomial number of iterations. In fact, as a consequence of more nuanced results in complexity theory, Theorem~\ref{thm:eps.approx.cubic} further implies that such behavior must occur on a ``significant portion'' of cubic $n$-variate polynomials; see \citep[Corollary 2.2]{hemaspaandra2012sigact} for a precise complexity-theoretic statement.


We also claim that computing an $\epsilon$-approximate critical point remains hard even in the setting where a critical point of $p$ is \ghmra{guaranteed to exist.} Indeed, we can show that if there exists a polynomial-time algorithm which takes as input an $n$-variate cubic polynomial \ghmra{(with coefficients $\{c_{\alpha}\}_{\alpha}$)} with a promise that this polynomial has a critical point and returns a \ghmra{$\left(\max_{\alpha} |c_{\alpha}| \cdot n^{1/2-\eta}\right)$}-approximate critical point, \ghmra{for any fixed $\eta >0$,} then the \textsc{Planted Clique} conjecture~\citep{jerrum1992large,kuvcera1995expected} would be refuted. This well-known conjecture involves taking as input two positive integers, $m$ and $k \in \{1,\ldots,m\}$,  and creating an Erd\"os-R\'enyi graph on $m$ vertices by connecting each pair of vertices independently with probability $1/2$. A clique of size $k$ is then planted in this graph by randomly choosing a subset of $k$ \amirali{vertices} from the $m$ original \amirali{ones} and adding all edges between \amirali{them}. The conjecture, which is widely believed to be true, states that, unless P=NP, recovering this clique in polynomial time is impossible whenever $k=o(\sqrt{m})$ (and $k \geq (2+\epsilon)\log(m)$ for any constant $\epsilon>0$) \citep{chen2025almost}. We omit the proof of our claim above: it follows from the same reduction as in the proof of Theorem~\ref{thm:eps.approx.cubic}. \ghmra{Let
$\bar v=(\bar{x},
\bar{w}^{(1)},\ldots,\bar{w}^{(K)})$
be a $\left(\max_{\alpha} |c_{\alpha}| \cdot n^{1/2-\eta}\right)$-approximate critical point returned by the algorithm. Using this fact, \eqref{eq:square} and claim $(iii')$ in the proof of Theorem~\ref{thm:eps.approx.cubic}, we have
\[
K g(\bar{x})
\leq \|\nabla p(\bar v)\|^2
\leq \left(\max_{\alpha}|c_{\alpha}|\right)^2 n^{1-2\eta}
<K.
\]
Thus, $g(\bar{x})<1$. The proof of Lemma~\ref{prop:lb.g.s} implies that $\operatorname{sgn}(\bar{x})$ is the indicator vector of a clique of size $k$ in $G$.}


Finally, we briefly discuss the implications of Theorem~\ref{thm:eps.approx.cubic} for the design of higher-order Newton methods~\citep{nesterov2021implementable,silina2022unregularized,ahmadi2024higher,cartis2022evaluation}. Recall that in each iteration of the classical Newton method, one computes a critical point of the second-order Taylor expansion of the objective function around the current iterate. This computation is tractable, as it reduces to solving a linear system and can be performed in time polynomial in the dimension. Higher-order Newton methods aim to improve the convergence rate and oracle complexity of Newton’s method by instead working with Taylor expansions of degree 3 or higher. While recent work has proposed higher-order methods with tractable per-iteration complexity~\citep{nesterov2021implementable,ahmadi2024higher}, these algorithms are more involved than a direct extension of Newton’s method—namely, one that computes (approximately) a critical point of a higher-order Taylor expansion. Theorem~\ref{thm:eps.approx.cubic} shows that, unless P=NP, this most natural higher-order extension of Newton’s method lacks tractable per-iteration complexity already when the Taylor expansion has degree~3.

\subsection{Complexity of computing an $\epsilon$-approximate critical point of a lower bounded function} \label{subsec:eps.approx.quartic}

\aaa{In this section, we investigate whether computing an approximate critical point of a function remains intractable if the function is known to be lower bounded. Let us observe first that it is NP-hard to decide whether a lower bounded function has a critical point.}

\begin{proposition} \label{prop:crit.lb}
If there is a polynomial-time algorithm that can decide whether the exponential of a cubic polynomial (which is always lower bounded) has a critical point, then P=NP.
\end{proposition}

\begin{proof} 
A function $f:\mathbb{R}^n\rightarrow\mathbb{R}$ has a critical point if and only if $e^f$, which is lower bounded, has a critical point. Indeed, $\nabla e^{f(x)}=e^{f(x)} \nabla f(x)$. It is known that deciding whether a \aaa{cubic polynomial has a critical point is NP-hard~\citep{ahmadi2022complexity}, and this concludes the proof.}\end{proof}

\aaa{The next proposition shows that even} under additional assumptions, namely guaranteed existence of a critical point and absence of any spurious critical points\footnote{We say that a function has no spurious critical points if any critical point is a global minimum. Observe that absence of spurious critical points implies absence of spurious local minima.}, the problem of computing \aaa{a critical point} of a lower bounded function remains hard. 
\ghmr{To see this, let $\{\tilde{q}_i(x)=0, i=1,\ldots,r\}$ be a set of quadratic equations, for $x \in \mathbb{R}^n,$ which is known to have a (rational) solution, but for which finding a solution is intractable.} As an example, such a set of quadratic equations can encode \aaa{the} \textsc{Planted Clique} problem (see Section~\ref{subsec:eps.approx.no.lb}) \aaa{or} the problem of finding a Nash equilibrium in a bimatrix game\footnote{A Nash equilibrium can be defined by a set of quadratic inequalities which ensure no unilateral payoff improvement when moving to a pure strategy. One can then turn this into a set of quadratic equations by introducing slack variables.}. \ghmra{The proposition invokes the complexity class PPAD, which is the class of function problems in NP for which existence of a solution is guaranteed by a certain parity argument. Arguably, the best-known PPAD-complete problem is the problem of computing a Nash equilibrium; see \citep{daskalakis2009complexity,chen2006settling} for more context and the exact definition of the complexity class PPAD.}

\begin{proposition}\label{thm:critical.point.spurious}
If there is a polynomial-time algorithm that takes as input a lower bounded function (the product of an $n$-variate quartic polynomial and an exponential function in one variable), which is guaranteed to have critical points but no spurious critical points, and returns a critical point, then the \textsc{Planted Clique} conjecture \aaa{is} \amirali{refuted} and PPAD=P.
\end{proposition}

\begin{proof}
\amirali{As outlined above, let $\{\tilde{q}_i(x)=0, i=1,\ldots,r\}$}
be a set of quadratic equations \amirali{in $n$ variables} whose \amirali{solution(s) correspond} to the planted clique in \amirali{the setting of the \textsc{Planted Clique}  problem,} or to Nash \amirali{equilibria of} a bimatrix game. Define the function \amirali{$f:\mathbb{R}^{n+1}\rightarrow\mathbb{R}$ as}$$f(x,w)=e^{w} \cdot \left(\sum_{i=1}^r \tilde{q}_i^2(x)\right).$$ Clearly, $f$ is lower bounded  by zero. \amirali{Furthermore,} we have $$\nabla f(x,w)=\begin{pmatrix} 2e^w \cdot \sum_{i=1}^r \tilde{q}_i(x) \nabla \tilde q_i(x)\\  e^w (\sum_{i=1}^r \tilde{q}^2_i(x)) \end{pmatrix}.$$ Thus, any critical point $(\bar{x},\bar{w})$ of $f$ must be such that $\tilde{q}_i(\bar{x})=0$ for $i=1,\ldots,r$. 
\amirali{This implies that at any critical point $(\bar{x},\bar{w})$, we have $f(\bar{x},\bar{w})=0,$ and thus $f$ has no spurious critical points. It} also implies that any polynomial-time algorithm returning a critical point $(\bar{x},\bar{w})$ of $f$ would \amirali{provide us, through the vector $\bar{x}$, the planted clique of interest or a Nash equilibrium of the game in polynomial time. This establishes the two claims.}
\end{proof}
We remark that \aaa{Proposition~\ref{thm:critical.point.spurious}} goes against the somewhat commonly held belief that minimizing a function with no spurious local minima is tractable. We now study the complexity of finding an \emph{approximate} critical point for a lower bounded function. It is known that a lower bounded differentiable function always has an $\epsilon$-approximate critical point for any $\epsilon>0$ \citep{ekeland1974variational}. \ghmra{The problem of finding such a point is thus a total search problem, whose complexity is perhaps better studied through subsets of TFNP (the class of function problems in NP that are guaranteed to have an answer). For our next result, we consider a different question, which is also relevant from a computational perspective: can one efficiently find an $\epsilon$-approximate critical point whose norm is not unreasonably large? We show that the existence of a pseudo-polynomial-time algorithm with such a guarantee would imply P=NP.}

\begin{theorem}\label{thm:lower-bounded-exponential-radius}
\ghmra{Fix a scalar $\eta>0$.} If there is an algorithm
$\mathcal A$ that takes as input a \ghmra{lower bounded $n$-variate sextic
polynomial $p$ with coefficients
$\{c_\alpha\}_\alpha$}, runs in \ghmra{pseudo-polynomial time}, and is such
that
\begin{enumerate}[(i)]
    \item \ghmra{when $p$ does not have a critical point in
    $\{-1,0,1\}^n$}, then $\mathcal A$ is free to return an arbitrary
    point,
    \item \ghmra{when $p$ has a critical point in $\{-1,0,1\}^n$, then
    $\mathcal A$ returns a $\left(
    \max_\alpha |c_\alpha|\cdot n^{1/2-\eta}
    \right)$-approximate} critical point
    in the ball of radius $2^n$ centered at the origin,
\end{enumerate}
then $\mathrm{P}=\mathrm{NP}$.
\end{theorem}

Note that the guarantees required of algorithm $\mathcal{A}$ are very weak: when the input polynomial has an exact critical point in $\{-1,0,1\}^n$, the algorithm only needs to return a coarse approximate critical point in an exponentially large ball; on all remaining inputs, the algorithm may behave arbitrarily. 

\ghmra{The proof of Theorem \ref{thm:lower-bounded-exponential-radius} is based on a reduction from
\textsc{Clique} that ensures that, when the graph has no clique of
size $k$, all approximate critical points at the prescribed
accuracy lie outside a ball of exponential radius. To achieve this, we use the
following family of degree-six polynomials, parametrized by $t$,
whose gradient norm is bounded below by a certain quantity throughout a sufficiently large ball, and whose
gradient vanishes at the origin when $t=0$. In the reduction,
we will replace $t$ by the quartic polynomial $g(x)$ in \eqref{eq:def.g} encoding the
clique instance.

\begin{lemma}\label{lem:sextic-chain}
Let $C\geq 2$ and $R,L\geq 1$ be integers. Let $u=(u_1,\ldots,u_R)\in\mathbb{R}^R$ and
$z=(z_1,\ldots,z_L)\in\mathbb{R}^L$. For $t\in\mathbb{R}$, define
the degree-six polynomial
\[
c_t(u,z)
=
-t\sum_{j=1}^{R}u_j
+\sum_{j=1}^{R-1}(u_{j+1}-u_j)^6
+z_L^6
+(Cz_1-u_1)^6
+\sum_{i=1}^{L-1}(Cz_{i+1}-z_{i})^6.
\]
Then,
\begin{enumerate}[(i)]
    \item $\nabla c_0(0,0)=0$,
    \item for any $t\in\mathbb{R}$ and $(u,z)\in\mathbb{R}^{R+L}$
    satisfying $|z_L|^5\leq \frac{C^LR|t|}{12},$ we have $\|\nabla c_t(u,z)\|\geq\frac{\sqrt R\,|t|}{4}.$
\end{enumerate}
\end{lemma}


\begin{proof}[Proof of Lemma~\ref{lem:sextic-chain}.]
To see $(i),$ note that when $t=0$, $c_0(u,z)$ is a sum of sixth powers of linear forms, where each linear form is zero when $(u,z)=(0,0).$ Thus, every partial derivative of $c_0$ is equal to zero when $(u,z)=(0,0)$ and the conclusion follows.

We now prove $(ii)$. Define the vector $d=(d^u,d^z)\in\mathbb{R}^{R+L}$ by
\[
d_j^u=1,\quad j=1,\ldots,R,
\qquad
d_i^z=C^{-i},\quad i=1,\ldots,L.
\]
By definition, we have
\begin{equation} \label{eq:d.zero}
d_{j+1}^u-d_j^u=0,\quad j=1,\ldots,R-1,
\qquad
Cd_1^z-d_1^u=0,
\qquad
Cd_{i+1}^z-d_i^z=0,\quad i=1,\ldots,L-1.
\end{equation}
By the Cauchy-Schwarz inequality, we have
\begin{align} \label{eq:CS}
||\nabla c_t(u,z)|| \geq \frac{1}{||d||} \cdot |\langle \nabla c_t(u,z),d\rangle|.
\end{align}
We lower bound each term in the right-hand side of \eqref{eq:CS}. Since $C\geq2$ and $R\geq1$,
\[
\|d\|^2
=R+\sum_{i=1}^L C^{-2i}
\leq R+\sum_{i=1}^\infty 2^{-2i} =R+\frac{1/4}{1-1/4}
=R+\frac13
\leq4R,
\]
which implies that $1/||d|| \geq 1/(2\sqrt{R}).$ 
To lower bound the second term, we use the directional derivative of $c_t$ in the direction $d$. Let $s\in\mathbb{R}$. Using \eqref{eq:d.zero}, we get
\begin{align*}
c_t(u+sd^u,z+sd^z)
=c_t(u,z)+t\sum_{j=1}^R u_j
-t\sum_{j=1}^R(u_j+sd_j^u)
-z_L^6+(z_L+sd_L^z)^6.
\end{align*}
It follows that
\[
\langle\nabla c_t(u,z),d\rangle
=\left.\frac{d c_t(u+sd^u,z+sd^z)}{ds}\right|_{s=0}
=-t\sum_{j=1}^R d_j^u+6z_L^5d_L^z
=-Rt+6C^{-L}z_L^5,
\]
where we have used the definition of $d$ in the last equality. Thus, using the assumption of (ii),
\[
|\langle\nabla c_t(u,z),d\rangle|
\geq R|t|-6C^{-L}|z_L|^5
\geq R|t|-\frac{6C^{-L}C^LR|t|}{12}
=\frac{R|t|}{2}.
\]
In view of \eqref{eq:CS}, this inequality and the fact that $1/\|d\|\geq 1/(2\sqrt R)$ give the desired conclusion.
\end{proof}

\begin{proof}[Proof of Theorem \ref{thm:lower-bounded-exponential-radius}.]
We prove that such an algorithm would solve \textsc{Clique} in polynomial
time. In view of Remark~\ref{rem:eta}, we restrict ourselves to $\eta \in (0,1/2)$ throughout the proof. Let $G$ be a graph on $m$ vertices\footnote{We require here that $m \geq 23$. This is not a problem as \textsc{Clique} clearly remains hard on such instances.} with adjacency matrix $A$ and let $k\in\{1,\ldots,m\}$. Let 
\begin{align} \label{eq:def.C.T.R.L}
C=64,
\qquad
\theta=\left\lceil\frac{12}{\eta}\right\rceil,\qquad R=m^\theta, \qquad L=6R,
\end{align}
and introduce the variables $u=(u_1,\ldots,u_R)\in\mathbb{R}^R$ and $z=(z_1,\ldots,z_L)\in\mathbb{R}^L.$ Let $x=(x_1,\ldots,x_m)$ and recall the definition of $g(x)=g_{G,k}(x)$ in \eqref{eq:def.g} and the polynomial $c_t(u,z)$ from Lemma \ref{lem:sextic-chain}. We define the following sextic polynomial in
$n=m+R+L=m+7R$ variables:
\begin{equation} \label{eq:def.p.th3}
\begin{aligned}
p(x,u,z)
&=\sum_{i=1}^m(x_i^3-x_i)^2+c_{g(x)}(u,z)\\
&=\sum_{i=1}^m(x_i^3-x_i)^2
-g(x)\sum_{j=1}^R u_j
+\sum_{j=1}^{R-1}(u_{j+1}-u_j)^6
+z_L^6+(Cz_1-u_1)^6
+\sum_{i=1}^{L-1}(Cz_{i+1}-z_i)^6.
\end{aligned}
\end{equation}

Note that the reduction from $(G,k)$ to $p$ can be carried out in polynomial time. As $\eta$ is fixed, $R=m^{c}$, for some constant $c$, and $L=m^{c'}$, for some constant $c'$, and so $p$ has $n=m^{c''}$ variables and $m^{c'''}$ coefficients, for some constants $c''$ and $c'''$, each of which is an integer. Let $\gamma=\max_{\alpha}|c_\alpha|$ denote the largest absolute coefficient of $p$. By definition of $g$, all of its coefficients except for the constant term have absolute value at most $3m^2$. Its constant term is at most
$m^3+\frac{3}{2}m^2
\leq 2m^3$
and hence every coefficient of $g$ has absolute value at most $2m^3$. The coefficients of $\sum_{i=1}^m(x_i^3-x_i)^2$ have absolute value at most $2$, while the coefficients of $-g(x)\sum_{j=1}^R u_j$ have absolute value at most $2m^3$. The remaining terms in $p$ are sixth powers of linear forms. The coefficients of $(u_{j+1}-u_j)^6$ have magnitude at most $\binom{6}{3}=20$, while the coefficients arising from $(Cz_1-u_1)^6$ and $(Cz_{i+1}-z_i)^6$ have magnitude at most $C^6$. Taking into account possible overlap of the pure sixth-power monomials with adjacent terms, their resulting coefficients have magnitude at most
$C^6+1=2^{36}+1.$
Since $C=64$ and $m\geq 23$, we obtain
\begin{align} \label{eq:gamma.ub}
\gamma=\max_\alpha |c_\alpha|
\leq
\max\{2m^3,2^{36}+1\}
\leq m^8.
\end{align}
It follows that each coefficient of $p$ requires $O(\log m)$ bits to encode, which is why the reduction rules out not just polynomial-time algorithms, but also pseudo-polynomial-time ones.

We now argue that $p$ is lower bounded by showing that it is coercive. Let $p_6$ denote the degree-six homogeneous part of $p$, that is,
\[
p_6(x,u,z)=\sum_{i=1}^m x_i^6+\sum_{j=1}^{R-1}(u_{j+1}-u_j)^6
+z_L^6
+(Cz_1-u_1)^6
+\sum_{i=1}^{L-1}(Cz_{i+1}-z_i)^6.
\]
It is easy to see that $p_6$ is nonnegative. Furthermore, $p_6=0$ if and only if $(x,u,z)=0$. Indeed, if $p_6(x,u,z)=0$, then $x_i^6=0$ for all $i$, so $x=0$. The term $z_L^6$ and the terms $(Cz_{i+1}-z_i)^6$ then imply that $z=0$. The term $(Cz_1-u_1)^6$ and the terms $(u_{j+1}-u_j)^6$ further imply that $u=0$. Thus, $p_6$ is positive definite and so $p$ is coercive \cite[Section 4.2]{jeyakumar2014polynomial}.

We now establish the following three claims: $(i')$ if there is a clique of size $k$ in $G$, then $p$ has a critical point in $\{-1,0,1\}^n$; $(ii')$ if there is no clique of size $k$ in $G$, then $16\|\nabla p(x,u,z)\|^2\geq R$ for every $(x,u,z) \in\mathbb{R}^n$ with $\|(x,u,z)\|\leq 2^n$; and $(iii')$ $16\gamma^2n^{1-2\eta}<R$. Once we have shown these claims, the theorem follows. To see this, let $(\bar{x},\bar{u}, \bar{z}) \in\mathbb{R}^n$ be the point returned by algorithm $\mathcal{A}$. If $\| (\bar{x},\bar{u}, \bar{z})  \|^2\leq 2^{2n}$ and $16\|\nabla p(\bar{x},\bar{u}, \bar{z} )\|^2<R$, then there is a clique of size $k$ in $G$ per claim $(ii')$. Otherwise, $p$ has no critical point in $\{-1,0,1\}^n$ (otherwise the algorithm would have returned a point in the ball of radius $2^n$ with squared gradient norm at most $\gamma^2n^{1-2\eta}<R/16$, by claim $(iii')$). Per claim $(i')$, there is no clique of size $k$ in $G$. Since these comparisons can be performed exactly in polynomial time, if an algorithm like $\mathcal{A}$ existed, we would be able to solve \textsc{Clique} in polynomial time, which would imply $\mathrm{P}=\mathrm{NP}$.

We first prove claim $(i')$. Suppose $G$ has a clique of size $k$, and let $\bar x\in\{-1,1\}^m$ be its indicator vector. Further let $\bar u=\bar z=0$. We show that $(\bar x,\bar u,\bar z)\in\{-1,0,1\}^n$ is a critical point of $p$. We have
\[
\nabla p(x,u,z)=
\begin{pmatrix}
\nabla_x p(x,u,z)\\
\nabla_{u,z}p(x,u,z)
\end{pmatrix}
\]
where $\nabla_{u,z} p(x,u,z)=\nabla c_{g(x)}(u,z)$ and
\[
\frac{\partial p}{\partial x_i}(x,u,z)
=
2(x_i^3-x_i)(3x_i^2-1)
-\left(\sum_{j=1}^R u_j\right)
\frac{\partial g}{\partial x_i}(x),
\quad i=1,\ldots,m.\]
Since $g(\bar x)=0$ by Lemma~\ref{prop:clique.zeros.g} and $\bar x_i^3-\bar x_i=0$ for every $i$, it follows that $\nabla_x p(\bar x,0,0)=0.$
Furthermore,
\[
\nabla_{u,z}p(\bar x,0,0)
=\nabla c_{g(\bar x)}(0,0)
=\nabla c_0(0,0)=0,
\]
where the last equality follows from Lemma~\ref{lem:sextic-chain}(i). Thus, $(\bar x,\bar u,\bar z)$ is a critical point of $p$.

We next prove claim $(ii')$. Suppose that $G$ has no clique of size $k$. We show that
\begin{align}\label{eq:proof.claim.ii'}
16\|\nabla p(x,u,z)\|^2\geq R,
\quad \forall (x,u,z)\in\mathbb{R}^n
\text{ with }\|(x,u,z)\|\leq 2^n.
\end{align}
Let $(x,u,z)\in\mathbb{R}^n$ with $\|(x,u,z)\|\leq 2^n$. To leverage
Lemma~\ref{lem:sextic-chain} (ii), we first show that
\begin{align} \label{eq:ub.zL}
|z_L|^5\leq \frac{C^L Rg(x)}{12},
\end{align}
where $C, L,$ and $R$ are defined in \eqref{eq:def.C.T.R.L}.
As $||(x,u,z)||\leq 2^n$ and using the fact that $n=m+7R$, $|z_L|^5
\leq (2^n)^5
=
2^{5n}
=
2^{5m+35R}.$
Furthermore, since $G$ has no
clique of size $k$, Lemma~\ref{prop:lb.g.s} gives $g(x)\geq1$, and hence
\[
\frac{C^L Rg(x)}{12}
\geq
\frac{2^{6L}R}{12}
=
\frac{2^{36R}R}{12},
\]
where we have used $C=2^6$ and $L=6R$. Thus, it suffices to show that
$2^{5m+35R}
\leq
\frac{2^{36R}R}{12},$ or equivalently, $5m+\log_2 12
\leq
R+\log_2 R$ to obtain \eqref{eq:ub.zL}.
Since $\theta=\lceil 12/\eta\rceil$ and $\eta<1/2$, we have $\theta \geq25$.
Hence, for $m\geq2$, $R=m^\theta\geq m^{25}\geq 5m+4$
and, since $\log_2 12<4$, it follows that
$5m+\log_2 12
<
5m+4
\leq
R
\leq
R+\log_2 R.$
Therefore, \eqref{eq:ub.zL} holds.

We can now apply Lemma~\ref{lem:sextic-chain}(ii) to obtain
\[
\|\nabla p(x,u,z)\|
\geq
\|\nabla_{u,z}p(x,u,z)\|
=
\|\nabla c_{g(x)}(u,z)\|
\geq
\frac{\sqrt{R}\,g(x)}{4}
\geq
\frac{\sqrt{R}}{4},
\]
where the second inequality follows from
Lemma~\ref{lem:sextic-chain}(ii) and the third from
Lemma~\ref{prop:lb.g.s}. As a consequence,
$16\|\nabla p(x,u,z)\|^2\geq R$, proving claim $(ii')$.

Finally, to see claim $(iii')$, note that
\begin{align*}
16\gamma^2 n^{1-2\eta} \leq 16m^{16}(8m^\theta)^{1-2\eta} \leq 128m^{16+\theta(1-2\eta)} \leq 128m^{\theta-8}<m^\theta=R,
\end{align*}
where we used $\gamma\leq m^8$ and $n=m+7m^{\theta}\leq 8m^\theta$ for the first inequality, $\eta\in(0,1/2)$ for the second, $\theta\eta\geq 12$ for the third, and $m\geq 23$ for the last.
\end{proof}

}

Our final result shows that approximate computation of critical points remains hard if we change the notion of approximation from $\epsilon$-approximate critical to $\epsilon$-near critical (see Section \ref{sec:preliminaries}) and that this is the case even for lower-bounded functions. 


\begin{theorem} \label{thm:near.critical}
\ghmra{ Fix a scalar $\eta>0$ and a rational number $r>0$.} If there is an algorithm $\mathcal{A}$ that takes as input a lower-bounded function $f$, which is the product of an $n$-variate quartic polynomial and an exponential function in one variable, runs in \ghmra{pseudo-polynomial time}, and is such that:
    \begin{enumerate}[(i)]
    \item when $f$ has no critical points in \ghmra{$\{-r,0,r\}^{n+1}$}, $\mathcal{A}$ is free to return an arbitrary point,
    \item when $f$ has a critical point in \ghmra{$\{-r,0,r\}^{n+1}$ , $\mathcal{A}$ returns a $r(n+1)^{1/2-\eta}$-near} critical point,
    \end{enumerate}
    then P=NP.
    \end{theorem}

    \ghmra{

    \begin{remark}\label{rem:tight.near}
Similarly to Remark \ref{rem:tight.with.n}, two features of the approximation factor appearing in Theorem~\ref{thm:near.critical} are worth noting. First, the dependence of the approximation factor on $n$ is essentially tight. Indeed, if $f$ has a critical point $x^* \in \{-r,0,r\}^{n+1}$, then the origin is always an $r\sqrt{n+1}$-near critical point of $f$ as $||x^*|| \leq r \sqrt{n+1}.$ Theorem~\ref{thm:near.critical} shows that, unless P=NP, improving this trivial guarantee to
$r (n+1)^{1/2-\eta},$ for any fixed $\eta>0$, is computationally intractable. Second, the approximation factor is expressed relative to $r$, rather than in absolute terms. This is because, for any differentiable function $f$ and any $s>0$,
replacing $f(x)$ by $f(x/s)$ sends every critical point $x^*$ of $f$
to the critical point $sx^*$. Thus, $\bar x$ is an
$\epsilon$-near critical point of $f$ if and only if $s \bar x$ is an $s\epsilon$-near critical point of the rescaled function $f(x/s)$. Thus, as with approximate critical points, an absolute hardness guarantee can be amplified by rescaling the instance. However, in applications one may know a priori that the critical points lie in a region of controlled scale. Expressing the approximation factor relative to $r$ accounts for this scale and yields a normalized statement.
\end{remark}

\begin{proof}[Proof of Theorem~\ref{thm:near.critical}.]
We prove that such an algorithm would solve \textsc{Clique} in polynomial
time. In view of Remark~\ref{rem:eta}, we restrict ourselves to $\eta \in (0,1/2)$ for the proof. Let $G$ be a graph on $m\geq 2$ vertices with adjacency matrix $A$,
and let $k\in\{1,\ldots,m\}$. Recall the definition of $g=g_{G,k}$ in
\eqref{eq:def.g}. Let $x=(x_1,\ldots,x_m)$ and for the given scalar $r>0$, define
the scaled quartic polynomial $g^{(r)}(x) =r^4 g(x/r)$, i.e., 
\begin{equation} \label{eq:def.gr}
g^{(r)}(x)
=
r^2\left(\sum_{i=1}^m x_i-r(2k-m)\right)^2
+
m^2\sum_{i=1}^m(x_i^2-r^2)^2
+
\sum_{1\leq i<j\leq m}
(1-A_{ij})^2(x_i+r)^2(x_j+r)^2.
\end{equation}
Note that $g^{(r)}(x)=0$ if and only if $g(x/r)=0.$
In particular, by Lemma~\ref{prop:clique.zeros.g}, $G$ has a clique of
size $k$ if and only if $g^{(r)}$ has a zero in $\{-r,r\}^m$. Define $\theta=\left\lceil\frac{1}{\eta}\right\rceil$ and $K=m^\theta.$
Let $w \in \mathbb{R}$ be a new variable. For $\ell=1,\ldots,K$, introduce disjoint blocks of variables
$x^{(\ell)}\in\mathbb{R}^m$, and define
\[
f\left(x^{(1)},\ldots,x^{(K)},w\right)
=
e^w p\left(x^{(1)},\ldots,x^{(K)}\right),
\text{ where } 
p\left(x^{(1)},\ldots,x^{(K)}\right)
=
\sum_{\ell=1}^K g^{(r)}\left(x^{(\ell)}\right).
\]
The polynomial $p$ is quartic and has $n=Km$ variables.

Note that the reduction from $(G,k)$ to $f$ can be carried
out in polynomial time. Since $\eta$ is fixed, $K=m^{c}$ for some constant $c$, and hence
$n=m^{c'}$, for some constant $c'$. Moreover, $p$ consists of $K$ copies of $g^{(r)}$ on
disjoint blocks of variables. Since $r$ is fixed and rational, $g^{(r)}$ has $m^{c''}$ coefficients, for some constant $c''$, each of magnitude
$m^{c'''}$, for some constant $c''',$ and bit length $O(\log m)$, and the same holds for $p$. Thus, each coefficient of $f$ requires $O(\log(m))$ bits to encode, which is the property needed to rule out pseudo-polynomial-time algorithms.

It is straightforward to see that $f$ is lower bounded. Since $g$ is a sum of
squares, $g^{(r)}(x)=r^4g(x/r)$ is nonnegative, and therefore $p\left(x^{(1)},\ldots,x^{(K)}\right)$ is also nonnegative. As $e^w$ is positive, it follows that $f\left(x^{(1)},\ldots,x^{(K)},w\right)$ is nonnegative.
The gradient of $f$ is
\[
\nabla f\left(x^{(1)},\ldots,x^{(K)},w\right)
=
e^w
\begin{pmatrix}
\nabla g^{(r)}(x^{(1)})\\
\vdots\\
\nabla g^{(r)}(x^{(K)})\\
p(x^{(1)},\ldots,x^{(K)})
\end{pmatrix}.
\]
It is easy to see that if $G$ has a clique of size $k$ and $x^* \in \{-1,1\}^m$ is its indicator vector, then $(rx^*,\ldots,rx^*,0) \in \{-r,0,r\}^{n+1}$ is a critical point of $f$. Indeed, $g^{(r)}$ is nonnegative with $g^{(r)}(rx^*)=r^4g(x^*)=0$, thus $rx^*$ is a global minimum of $g^{(r)}$, which implies that $\nabla g^{(r)}(rx^*)=0.$ It then follows from the expression of the gradient of $f$ that $\nabla f(rx^*,\ldots,rx^*,0)=0.$

Let $\alpha =r(n+1)^{1/2-\eta}.$ We claim that if $(x^{(1)},\ldots,x^{(K)},w)$ is an $\alpha$-near critical point of $f$, then there exists $\ell^* \in \{1,\ldots,K\}$ such that $sgn(x^{(\ell^*)})$ is the indicator vector of a clique of size $k$.  Once we have shown this claim, the theorem follows. Indeed, suppose $\bar{v}=\left(\bar x^{(1)},\ldots,\bar x^{(K)},\bar w\right)$ is the point returned by the algorithm $\mathcal{A}$\footnote{Note that $\mathcal{A}$ would write this point down in time polynomial in $m$, since, by assumption, it is a pseudo-polynomial-time algorithm and the coefficients of $f$ are integers bounded in magnitude by a polynomial in $m$.}. We can check in polynomial time whether $sgn(\bar{x}^{(\ell)})$ is the indicator vector of a clique of size $k$ in $G$ for any $\ell \in \{1,\ldots,K\}.$ If this is the case for some $\ell$, then $G$ clearly has a clique of size $k$. If this is not the case for any $\ell$, then $\bar v$ cannot be an $\alpha$-near critical point per the above claim. Thus $f$ has no critical point in $\{-r,0,r\}^{n+1}$ (otherwise the algorithm would have returned an $\alpha$-near critical point). From the argument in the previous paragraph, it follows that $G$ has no clique of size $k$. Thus, if an algorithm such as $\mathcal{A}$ existed, we would be able to solve \textsc{Clique} in polynomial time, which would imply P=NP. 

It remains to prove the claim, which we do now. Suppose $v=(x^{(1)}, \ldots, x^{(K)},w)$ is an $\alpha$-near critical point of $f$, i.e., $||v-v^*|| \leq \alpha,$
where $v^*=\left(x^{*(1)}, \ldots, x^{*(K)},w^*\right)$ is a critical point of $f$. Thus, $\nabla_w f\left(v^*\right)=0$, which implies that $p\left(x^{*(1)}, \ldots, x^{*(K)}\right)=0.$ As $$p(x^{*(1)},\ldots,x^{*(K)})=
\sum_{\ell=1}^K g^{(r)}(x^{*(\ell)})$$ and every term in this sum is nonnegative, it follows that $g^{(r)}(x^{*(\ell)})=0$, for all $\ell=1,\ldots,K$. This, in turn, implies, from \eqref{eq:def.gr}, that $x^{*(\ell)}\in\{-r,r\}^m$ and so $\frac{x^{*(\ell)}}{r}\in\{-1,1\}^m$ with $g(\frac{x^{*(\ell)}}{r})=0$. From the proof of Lemma~\ref{prop:clique.zeros.g}, we have that $\frac{x^{*(\ell)}}{r}$ is an indicator vector of a clique of size $k$ in $G$ for every $\ell.$
Furthermore,
\[
\sum_{\ell=1}^K
\left\| x^{(\ell)}-x^{*(\ell)}\right\|^2
\leq
\|v-v^*\|^2
\leq \alpha^2=
r^2(n+1)^{1-2\eta}.
\]
Hence there exists $\ell^*\in\{1,\ldots,K\}$ such that
\begin{align*}
\left\| x^{(\ell^*)}-x^{*(\ell^*)}\right\|
\leq
\frac{r(n+1)^{1/2-\eta}}{\sqrt K} &\leq 
\frac{r(2Km)^{1/2-\eta}}{\sqrt K}\\
&=
r\,2^{1/2-\eta}
m^{1/2-\eta-\theta \eta} \leq
r\,2^{1/2-\eta}m^{-1/2-\eta} \leq
r\,2^{-2\eta}<
r,
\end{align*}
where the second inequality uses $n+1\leq2Km$ as $n=Km$ and $m\geq 2$, the equality uses
$K=m^\theta$, the third inequality uses $\theta \eta\geq1$, and the fourth uses
$m\geq2$.
Since $\left\| x^{(\ell^*)}-x^{*(\ell^*)}\right\|<r$
and every coordinate of $x^{*(\ell^*)}$ belongs to $\{-r,r\}$, every
coordinate of $ x^{(\ell^*)}$ has the same sign as the corresponding
coordinate of $x^{*(\ell^*)}$. Hence
$sgn(x^{(\ell^*)})
=
\frac{x^{*(\ell^*)}}{r},$ which is the indicator vector of a clique of size $k$. This proves the claim and completes the proof.
 \end{proof}

}

\section*{Acknowledgments}
We would like to thank two anonymous reviewers, an anonymous Associate Editor, and the Department Editor for Optimization at \textit{Operations Research}, Samuel Burer, for their helpful and constructive feedback.

We also acknowledge the use of ChatGPT during the revision process. No AI tools were used in preparing the first version of this draft. In revising the manuscript, ChatGPT assisted us in deriving Lemma~\ref{lem:sextic-chain}, which we subsequently restructured and independently verified. We also used ChatGPT to review our proofs, which helped us identify and correct several typographical errors.

\setcitestyle{authoryear,round}
\bibliographystyle{plainnat}
\bibliography{sample}

@article{karp1975computational,
  title={On the computational complexity of combinatorial problems},
  author={Karp, Richard M},
  journal={Networks},
  volume={5},
  number={1},
  pages={45--68},
  year={1975},
  publisher={Wiley Online Library}
}

@inproceedings{chewi2023complexity,
  title={On the complexity of finding stationary points of smooth functions in one dimension},
  author={Chewi, Sinho and Bubeck, S{\'e}bastien and Salim, Adil},
  booktitle={International {C}onference on {A}lgorithmic {L}earning {T}heory},
  pages={358--374},
  year={2023},
  organization={{PMLR}}
}

@article{vavasis1993black,
  title={Black-box complexity of local minimization},
  author={Vavasis, Stephen A},
  journal={{SIAM} Journal on Optimization},
  volume={3},
  number={1},
  pages={60--80},
  year={1993},
  publisher={{SIAM}}
}

@article{grigor1988solving,
  title={Solving systems of polynomial inequalities in subexponential time},
  author={Grigor'ev, D Yu and Vorobjov Jr, Nicolai N},
  journal={Journal of Symbolic Computation},
  volume={5},
  number={1-2},
  pages={37--64},
  year={1988},
  publisher={Elsevier}
}

@article{daskalakis2009complexity,
  title={The complexity of computing a {N}ash equilibrium},
  author={Daskalakis, Constantinos and Goldberg, Paul W and Papadimitriou, Christos H},
  journal={Communications of the ACM},
  volume={52},
  number={2},
  pages={89--97},
  year={2009},
  publisher={ACM New York, NY, USA}
}

@inproceedings{chen2006settling,
  title={Settling the Complexity of Two-Player {N}ash Equilibrium.},
  author={Chen, Xi and Deng, Xiaotie},
  booktitle={Foundations of Computer Science},
  volume={6},
  pages={261--272},
  year={2006}
}

@article{hollender2025computational,
  title={The computational complexity of finding stationary points in non-convex optimization},
  author={Hollender, Alexandros and Zampetakis, Manolis},
  journal={Mathematical Programming},
  volume={213},
  number={1},
  pages={281--341},
  year={2025},
  publisher={Springer}
}

@article{fearnley2025complexity,
  title={The complexity of computing {KKT} solutions of quadratic programs},
  author={Fearnley, John and Goldberg, Paul and Hollender, Alexandros and Savani, Rahul},
  journal={Journal of the ACM},
  volume={72},
  number={5},
  pages={1--48},
  year={2025},
  publisher={ACM New York, NY}
}

@inproceedings{kontogiannis2024computational,
  title={The computational complexity of finding second-order stationary points},
  author={Kontogiannis, Andreas and Pollatos, Vasilis and Kanellopoulos, Sotiris and Mertikopoulos, Panayotis and Pagourtzis, Aris and Panageas, Ioannis},
  booktitle={Forty-first International Conference on Machine Learning},
  publisher={Proceedings of Machine Learning Research},
  volume={235},
  pages={25191–25225},
  year={2024}
}

@article{ekeland1974variational,
  title={On the variational principle},
  author={Ekeland, Ivar},
  journal={Journal of Mathematical Analysis and Applications},
  volume={47},
  number={2},
  pages={324--353},
  year={1974},
  publisher={Elsevier}
}

@book{bertsekas1997nonlinear,
  title={Nonlinear {P}rogramming},
  author={Bertsekas, Dimitri P},
  year      = 1995,
  publisher = "Athena Scientific",
  address   = "Belmont"
}

@article{jeyakumar2014polynomial,
  title={On polynomial optimization over non-compact semi-algebraic sets},
  author={Jeyakumar, Vaithilingam and Lasserre, Jean B and Li, Guoyin},
  journal={Journal of Optimization Theory and Applications},
  volume={163},
  pages={707--718},
  year={2014},
  publisher={Springer}
}

@article{jerrum1992large,
  title={Large cliques elude the {M}etropolis process},
  author={Jerrum, Mark},
  journal={Random Structures \& Algorithms},
  volume={3},
  number={4},
  pages={347--359},
  year={1992},
  publisher={Wiley Online Library}
}

@article{kuvcera1995expected,
  title={Expected complexity of graph partitioning problems},
  author={Ku{\v{c}}era, Lud{\v{e}}k},
  journal={Discrete Applied Mathematics},
  volume={57},
  number={2-3},
  pages={193--212},
  year={1995},
  publisher={Elsevier}
}

@article{carmon2020lower,
  title={Lower bounds for finding stationary points {I}},
  author={Carmon, Yair and Duchi, John C and Hinder, Oliver and Sidford, Aaron},
  journal={Mathematical Programming},
  volume={184},
  number={1},
  pages={71--120},
  year={2020},
  publisher={Springer}
}

@inproceedings{daskalakis2021complexity,
  title={The complexity of constrained min-max optimization},
  author={Daskalakis, Constantinos and Skoulakis, Stratis and Zampetakis, Manolis},
  booktitle={Proceedings of the 53rd Annual {ACM SIGACT} Symposium on Theory of Computing},
  pages={1466--1478},
  year={2021}
}

@article{ahmadi2022complexityqp,
  title={On the complexity of finding a local minimizer of a quadratic function over a polytope},
  author={Ahmadi, Amir Ali and Zhang, Jeffrey},
  journal={Mathematical Programming},
  volume={195},
  number={1},
  pages={783--792},
  year={2022},
  publisher={Springer}
}

@article{etessami2010complexity,
  title={On the complexity of {N}ash equilibria and other fixed points},
  author={Etessami, Kousha and Yannakakis, Mihalis},
  journal={SIAM Journal on Computing},
  volume={39},
  number={6},
  pages={2531--2597},
  year={2010},
  publisher={SIAM}
}

@article{murty1985some,
  author  = {Murty, Katta G. and Kabadi, Santosh N.},
  title   = {Some {NP}-complete problems in quadratic and nonlinear programming},
  journal = {Mathematical Programming},
  volume  = {39},
  number  = {2},
  pages   = {117--129},
  year    = {1987}
}

@book{cartis2022evaluation,
  title={Evaluation {C}omplexity of {A}lgorithms for {N}onconvex {O}ptimization: {T}heory, {C}omputation and {P}erspectives},
  author={Cartis, Coralia and Gould, Nicholas IM and Toint, Philippe L},
  year={2022},
  publisher={SIAM}
}

@book{nocedal2006numerical,
  title={Numerical {O}ptimization},
  author={Nocedal, Jorge and Wright, Stephen J},
  year={2006},
  publisher={Springer}
}

@book{boyd2004convex,
  title={Convex {O}ptimization},
  author={Boyd, Stephen P and Vandenberghe, Lieven},
  year={2004},
  publisher={Cambridge {U}niversity {P}ress}
}

@article{carmon2021lowerii,
  title={Lower bounds for finding stationary points {II}: {F}irst-order methods},
  author={Carmon, Yair and Duchi, John C and Hinder, Oliver and Sidford, Aaron},
  journal={Mathematical Programming},
  volume={185},
  number={1},
  pages={315--355},
  year={2021},
  publisher={Springer}
}

@article{fearnley2022complexity,
  title={The complexity of gradient descent: {CLS}= {PPAD $\cap$ PLS}},
  author={Fearnley, John and Goldberg, Paul and Hollender, Alexandros and Savani, Rahul},
  journal={Journal of the {ACM}},
  volume={70},
  number={1},
  pages={1--74},
  year={2022},
  publisher={ACM New York, NY}
}

@article{silina2022unregularized,
  title={An unregularized third order {N}ewton method},
  author={Silina, Olha and Zhang, Jeffrey},
  journal={arXiv preprint arXiv:2209.10051},
  year={2022}
}

@article{nesterov2021implementable,
  title={Implementable tensor methods in unconstrained convex optimization},
  author={Nesterov, Yurii},
  journal={Mathematical Programming},
  volume={186},
  number={1},
  pages={157--183},
  year={2021},
  publisher={Springer}
}

@article{ahmadi2024higher,
  title={Higher-order {N}ewton methods with polynomial work per iteration},
  author={Ahmadi, Amir Ali and Chaudhry, Abraar and Zhang, Jeffrey},
  journal={Advances in Mathematics},
  volume={452},
  pages={109808},
  year={2024},
  publisher={Elsevier}
}

@article{hemaspaandra2012sigact,
  title={An atypical survey of typical-case heuristic algorithms},
  author={Hemaspaandra, Lane A and Williams, Ryan},
  journal={ACM SIGACT News},
  volume={43},
  number={4},
  pages={70--89},
  year={2012},
  publisher={ACM New York, NY, USA}
}

@article{ahmadi2022complexity,
  title={Complexity aspects of local minima and related notions},
  author={Ahmadi, Amir Ali and Zhang, Jeffrey},
  journal={Advances in Mathematics},
  volume={397},
  pages={108119},
  year={2022},
  publisher={Elsevier}
}

@inproceedings{ge2015escaping,
  title={Escaping from saddle points—online stochastic gradient for tensor decomposition},
  author={Ge, Rong and Huang, Furong and Jin, Chi and Yuan, Yang},
  booktitle={Conference on Learning Theory},
  pages={797--842},
  year={2015},
  organization={PMLR}
}

@article{chen2025almost,
  title={Almost-Linear Planted Cliques Elude the {M}etropolis Process},
  author={Chen, Zongchen and Mossel, Elchanan and Zadik, Ilias},
  journal={Random Structures \& Algorithms},
  volume={66},
  number={2},
  pages={e21274},
  year={2025},
  publisher={Wiley Online Library}
}

@book{sipser1996introduction,
  author    = {Sipser, Michael},
  title     = {Introduction to the {T}heory of {C}omputation},
  edition   = {3rd},
  publisher = {Cengage Learning},
  year      = {2012}
}

\end{document}